\documentclass[12pt,a4paper]{article}
\usepackage[intlimits]{amsmath}
\usepackage{amsfonts,amssymb,amscd,amsthm}
\usepackage{cite}

\usepackage[all]{xy}

\def\lra{\longrightarrow}
\def\ind{\operatorname{ind}}

\def\Td{\operatorname{Td}}

\def\tr{\operatorname{tr}}

\def\End{\operatorname{End}}

\def\dim{\operatorname{dim}}

\newtheorem{theorem}{Theorem}[section]
\newtheorem{lemma}{Lemma}[section]
\newtheorem{proposition}{Proposition}[section]

\theoremstyle{definition}

\newtheorem{remark}{Remark}[section]

\numberwithin{equation}{section}

\title{On the Index Formula for an Isometric Diffeomorphism}
\author{A. Savin, E. Schrohe, B. Sternin}
\date{}

\begin{document}

\maketitle
\begin{abstract}
We give an elementary solution of the index problem for elliptic operators associated with the
shift operator along the trajectories of an isometric diffeomorphism of a closed smooth manifold. This solution is based on a reduction (which preserves the index) of the operator to an elliptic pseudodifferential operator
on a manifold of a higher dimension and an application of an Atiyah--Singer type formula. The final index formula is given in terms of the symbol of the operator on the original manifold.
\end{abstract}

\section*{Introduction}

Elliptic theory for operators with shifts  associated with group actions on smooth manifolds
is developing intensively in recent years.  
This theory (especially in its analytical aspects) goes back to the works of 
Antonevich and his coworkers, see~\cite{Ant2,AnLe1,AnLe2} 
and the references there. 
In these works finiteness theorems (Fredholm property) were obtained in sufficient generality, 
and there emerged the problem of finding a formula for the index in terms of the topological invariants of the symbol of the operator and the manifold, on which this operator is defined.  
First attempts to solve this problem were also made in the papers of Antonevich, 
at least,  for finite groups  (see the papers cited above).
In a number of special cases (namely for operators on the noncommutative torus, etc.) 
index formulas were obtained, for instance, in the papers   \cite{Con4,CoMo2,Mos2,Per3}. 
In the general case, when a discrete group of a quite general form is acting on a manifold isometrically, the index formula was established in the book   \cite{NaSaSt17}. 
In the nonisometric case, for the group generated
by a diffeomorphism of a manifold, the index formula was shown in a recent paper  \cite{SaSt30}.
Note that, in all these papers, the proofs of the  index formul\ae\ were highly nontrivial  and
involved complicated mathematical results, notions and constructions in
noncommutative geometry and algebraic topology.

Recently, the authors (see \cite{SaSchSt1, SaSchSt2})  used the idea of
uniformization and succeeded in reducing the index  problem for an
operator with shifts to a similar problem for a  {\em pseudo\-dif\-ferential}
operator on a manifold of a higher dimension. The index of the latter operator
is computed by the well-known Atiyah--Singer formula. This approach is
attractive, because it is quite elementary and does not need the 
complicated mathematical techniques mentioned above.  However, the
question of finding an index formula for the original operator in terms of its
symbol on the original manifold was left open in the cited paper. Here we fill this gap and
give an elementary solution of the index
problem in the considered class of operators.

We now briefly describe the contents of the paper.

In Section 1 we formulate the main result: we  express  the index of
an elliptic operator with shifts,   associated with an isometric diffeomorphism of a smooth
manifold, in terms of topological invariants of the operator's symbol. Note  that this index
formula is given in terms of the noncommutative geometry of A.~Connes  \cite{Con1}.

Sections 2-4 are devoted to the proof of the index formula. In
Section 2 we
briefly recall the main results of the papers \cite{SaSchSt1,SaSchSt2}.
Namely, to an elliptic operator with shifts on a manifold we assign an elliptic differential operator
on a different closed manifold   (actually, this manifold is the mapping torus of the diffeomorphism,
which defines the shift).
The index of the operator  obtained is calculated using an Atiyah--Singer type index formula
(cf.~\cite{AtSi3}). Thus, to prove our index formula, it remains to transform an Atiyah--Singer
type  formula to a noncommutative-geometric formula. This is done in Sections  3 and 4.
Without going into the technicalities of this transformation, let us note that Section 4 is crucial here:
we introduce a large parameter in the Atiyah--Singer formula\footnote{This large parameter is easy 
to describe
geometrically. Namely, it is equal to the radius of the cosphere bundle, over which we
integrate in the Atiyah--Singer formula.} and the noncommutative geometric formula is obtained as the
parameter tends to infinity.  The technique we apply is
conceptually close to asymptotic homomorphisms, which are widely used in noncommutative geometry and elliptic
theory  (e.g., see \cite{CoHi4,Man1,NaSaSt3}).

\section{The Index Theorem} \label{sect1}

\subsection{The Analytical Index}

Let $M$ be a smooth manifold and $g:M\to M$ be a diffeomorphism. We assume that
$g$ is an isometry, that is, it preserves a Riemannian metric on $M$.

Consider the operator
\begin{equation}\label{eq-op1}
D=\sum_k D_k T^k:C^\infty(M)\lra C^\infty(M),\qquad Tu(x)=u(g(x)),
\end{equation}
where $D_k$ are (for simplicity) differential operators of order one on  $M$.

The {\em symbol} of the operator $D$  (see e.g. \cite{AnLe2}) is a  function on
the cotangent bundle   $T^*M$, whose value at the point $(x,\xi)\in T^*M$ is
the bounded operator
\begin{equation}\label{traj-symb1}
\sigma(D)(x,\xi)=\sum_k \sigma(D_k)(\partial g^{n}(x,\xi))\mathcal{T}^k:l^2(\mathbb{Z})\lra l^2(\mathbb{Z}),
\end{equation}
which acts on the space of square summable sequences. Here
$\mathcal{T}u(n)=u(n+1)$ is the shift operator on the space of sequences, and
$\partial g=({}^tdg)^{-1}:T^*M\to T^*M$ is the codifferential of the
diffeomorphism $g$.

The operator  $D$ is {\em elliptic} if the symbol $\sigma(D)(x,\xi)$ is invertible for all 
$x$ and $\xi\ne 0$.
\begin{proposition}
An elliptic operator $D$ defines a Fredholm operator $D:H^s(M)\to H^{s-1}(M)$
in the Sobolev spaces.
\end{proposition}
\begin{proof}
If $g$ acts topologically freely, then the Fredholm property was proved in
\cite{AnLe2}. 
The proof in the general case  is similar (see e.g.\cite{NaSaSt17,SaSt23}).
\end{proof}

Operators of the form \eqref{eq-op1} are called {\em operators with shifts}.

\subsection{The Topological Index}

To describe the basic properties of the symbol, we introduce the subalgebra
\begin{equation*}\label{eq-alg2}
\mathcal{S}=\{A=(a_{ij}):l^2(\mathbb{Z})\to l^2(\mathbb{Z})\;|\; \forall N>1\;
\exists C_N:\;|a_{ij}|\le C_N(1+|i-j|)^{-N}\}
\end{equation*}
in the algebra of operators acting on the space $l^2(\mathbb{Z})$.

\begin{proposition}
\begin{enumerate}
\item The algebra $\mathcal{S}$ is a Fr\'echet algebra with respect to the
family of seminorms
$$
 \|a\|_N=\sup_{ij} |a_{ij}| (1+|i-j|)^{N}
$$
and is a local subalgebra in the algebra of bounded operators acting in  $l^2(\mathbb{Z})$.
\item The symbol  \eqref{traj-symb1} defines a smooth  $\mathcal{S}$-valued function
$$
\sigma(D) \in C^\infty(S^*M,\mathcal{S})
$$
on the cosphere bundle   $S^*M$.
\item The symbol is invariant under the action
\begin{equation}\label{eq-mu5}
 u(x,\xi) \longmapsto  \mathcal{T}^{-n}u(\partial g^n(x,\xi))\mathcal{T}^{n}, \quad n\in\mathbb{Z},
\end{equation}
of the group    $\mathbb{Z}$ on the algebra $C^\infty(S^*M,\mathcal{S})$.
The invariance property can be written  in terms of   matrix coefficients as
$$
 \sigma_{i,j}(D)(\partial g(x,\xi))=\sigma_{i+1,j+1}(D)(x,\xi).
$$
\end{enumerate}
\end{proposition}

\begin{proof}
Actually, we have an isomorphism   $\mathcal{S}\simeq l^\infty(\mathbb{Z})\rtimes \mathbb{Z}$
of the algebra $\mathcal{S}$  and the smooth crossed product 
of the algebra
$l^\infty(\mathbb{Z})$ of bounded sequences and the group   $\mathbb{Z}$ under the action of
$\mathbb{Z}$ by translations of sequences. Thus, we can apply the results of the paper
\cite{Schwe1} and obtain the first statement. The remaining statements are straightforward.
\end{proof}

Provided that the operator $D$ is elliptic,  its symbol is invertible and the inverse symbol
$$
 \sigma(D)^{-1}\in C^\infty(S^*M,\mathcal{S})^\mathbb{Z}
$$
exists. Here $C^\infty(S^*M,\mathcal{S})^\mathbb{Z}$  stands for the subspace of $\mathbb{Z}$-invariant
symbols. We now define a cycle in the sense of Connes \cite{Con1}
over the algebra  $\mathcal{A}=C^\infty(S^*M,\mathcal{S})^\mathbb{Z}$. The components of this cycle are:
\begin{enumerate}
\item $\Omega=\Omega(S^*M,\mathcal{S})^\mathbb{Z}$ is the algebra of $\mathbb{Z}$-invariant differential forms taking values in $\mathcal{S}$.
\item $d:\Omega\to \Omega $ is the differential, which is defined componentwise
$$
d\left(\omega_{ij}\right)=(d\omega_{ij}), \quad \omega=(\omega_{ij})\in\Omega.
$$
\item $\tau:\Omega\lra \mathbb{C}$ is the functional
\begin{equation}\label{eq-sled1}
\tau\left( \omega\right)=\int_{S^*M} \omega_{00} \cdot \Td(T^*_\mathbb{C}M),\quad \omega\in\Omega,
\end{equation}
where $ \Td(T^*_\mathbb{C}M)$ is a  $g$-invariant differential form on  $M$ that represents the Todd class
of the complexified cotangent bundle.
\end{enumerate}

\begin{proposition}
The functional \eqref{eq-sled1} is a closed graded trace, that is, it satisfies the relations
$$
\tau (d\omega)=0,\quad \tau(\omega_1\omega_2)=(-1)^{\deg \omega_1\deg\omega_2} \tau(\omega_2\omega_1)
$$
for all forms $\omega,\omega_1,\omega_2\in \Omega$.
\end{proposition}

\begin{proof}
It suffices to prove the trace property. Given   $\omega',\omega''\in \Omega$,  we have:
\begin{multline*}
 \tau(\omega'\omega'')=\sum_i\int_{S^*M} \omega'_{0i}\omega_{i0}'' \Td(T^*_\mathbb{C}M)=
\sum_i\int_{S^*M} \partial g^{-i*} (\omega'_{0i}\omega_{i0}'') \Td(T^*_\mathbb{C}M)=\\
=(-1)^{\deg\omega'\deg\omega''}\sum_i\int_{S^*M}   (\omega''_{0i}\omega_{i0}') \Td(T^*_\mathbb{C}M)=
(-1)^{\deg\omega'\deg\omega''}\tau   (\omega'' \omega ').
\end{multline*}
\end{proof}

The cycle over    $\mathcal{A}$ defines (e.g., see \cite{Con1}) the functional
\begin{equation}\label{eq-indt1}
 \begin{array}{ccc}
  \ind_t: K_1(\mathcal{A})  & \longrightarrow & \mathbb{C} \\
  a & \longmapsto & \sum_j C_j\tau(a^{-1}da)^{2j-1}
 \end{array}
\end{equation}
on the $K_1$-group of the algebra. Here $C_j=\frac{(j-1)!}{(2\pi i)^j(2j-1)!}$.
This functional is called the  {\em topological index}.

\subsection{The Index Theorem}
\begin{theorem}\label{th1}
The index $\ind D$ of an elliptic operator  $D$ with shifts on a manifold  $M$ is equal to
\begin{equation}\label{eq-index1}
 \ind D=\ind_t [\sigma(D)],
\end{equation}
where $[\sigma(D)]\in K_1(C^\infty(S^*M,\mathcal{S})^\mathbb{Z}) $ is the class of the symbol in  $K$-theory
and the topolo\-gical index $\ind_t [\sigma(D)]$ is defined in \eqref{eq-indt1}.
\end{theorem}

An index formula of this type was first obtained in   \cite{NaSaSt17}.
However, the proof was quite complicated. In particular, it substantially used
$K$-theory of crossed products (the periodicity theorem and the direct image mapping in
equivariant $K$-theory).

Here we give a direct proof of the index formula based on the ideas of uniformization.
Using these ideas, an index formula for  $D$ was obtained in the preceding paper  \cite{SaSchSt1}
in terms of the topological index of a special elliptic pseudodifferential operator   ($\psi$DO).
To prove  Theorem~\ref{th1}   we transform the topological index defined in
 \cite{SaSchSt1} to the form $\ind_t [\sigma(D)]$. 
For this purpose we recall in the next section the main results
of the cited paper, and in the subsequent sections we prove that the  topological indices are equal.

\section{Pseudodifferential Uniformization}

We recall the main results of the paper  \cite{SaSchSt1} in a form convenient for us.

\subsection{Reduction to a $\psi$DO}

Given an elliptic operator $D$ of the form  \eqref{eq-op1}, we extend the isometry
$g:M\to M$ to an isometry of the infinite cylinder  $M\times \mathbb{R}$ with local coordinates $x$ and $t$
$$
\begin{array}{ccc}
\widetilde{g}: M\times \mathbb{R} & \lra & M\times \mathbb{R},\\
(x,t) & \longmapsto & (g(x),t+1),
\end{array}
$$
and define  the associated shift operator   $ (\widetilde{T}u)(x,t)=u(g(x),t+1)$. Consider the operator
$$
 \widetilde{D}=\sum_k D_k \widetilde{T}^k:C^\infty(M\times \mathbb{R})\lra C^\infty(M\times \mathbb{R}),
$$
on the product $M\times \mathbb{R}$  and the operator
$$
 A=\frac\partial{\partial  t}+t:C^\infty(\mathbb{R})\to C^\infty(\mathbb{R}),
$$
on the line.   Define the  {\em external product} of the operators $\widetilde{D}$ and $A$:
\begin{equation}\label{eq-product1}
\widetilde{D}\# A=\left(
                   \begin{array}{cc}
                        \widetilde{D} &  A\\
                       -  A^* &   \widetilde{D}^*
                   \end{array}
                \right):C^\infty(M\times\mathbb{R},\mathbb{C}^2)\lra C^\infty(M\times\mathbb{R},\mathbb{C}^2).
\end{equation}
Here the adjoint operators are taken with respect to  the inner product on   $L^2(M\times \mathbb{R})$.

The isometry $\widetilde{g}$ defines a free proper action of the group  $\mathbb{Z}$ on the cylinder.
Thus, the corresponding orbit space is a smooth manifold. We will treat functions on the cylinder as
functions on the orbit space taking values in functions on the fibers of the projection
$$
M\times \mathbb{R}\lra  (M\times\mathbb{R})/\mathbb{Z}=M_\mathbb{Z}.
$$

Let $\mathcal{E}$  be the vector bundle
$$
\mathcal{E}=(M\times \mathbb{R}\times l^2(\mathbb{Z}))/ \mathbb{Z}
$$
over $M_\mathbb{Z}$ with the fiber $l^2(\mathbb{Z})$. Here the group
$\mathbb{Z}$ acts as
\begin{equation}\label{eq-mumu1}
(x,t,u)\longmapsto (g^n(x),t+n,\mathcal{T}^{-n}u),\qquad \mathcal{T}u(k)=u(k+1),\quad n\in \mathbb{Z}.
\end{equation}
Denote by
\begin{equation}\label{eq-op66}
 \mathcal{D}:C^\infty(M_\mathbb{Z},\mathcal{E}) \to C^\infty(M_\mathbb{Z},\mathcal{E})
\end{equation}
a differential operator with the operator-valued symbol
\begin{equation}\label{orb-symbol1}
\sigma(\mathcal{D})(x,\xi,t,\tau)=(\sigma(D)(x,\xi))\# (i\tau+t+n):l^2(\mathbb{Z},\mu_{\xi,\tau,s})\lra l^2(\mathbb{Z},\mu_{\xi,\tau,s-1}).
\end{equation}
Here the measure is defined as   $\mu_{\xi,\tau,s}(n)=(\xi^2+\tau^2+n^2)^s$, and the space $l^2$
consists of the sequences square summable with respect to this measure.
Note that the symbol  \eqref{orb-symbol1} is not homogeneous in the covariables   $\xi,\tau$.

The operator \eqref{eq-op66} is  {\em elliptic},  if its symbol \eqref{orb-symbol1}
is invertible for large  $\xi^2+\tau^2$, 
and the norm of the inverse symbol is uniformly bounded in   $x,t,\xi,\tau$.
It is proven in \cite{SaSchSt2}  that if the original operator   $D$ is elliptic, then
the differential operator $\mathcal{D}$ is elliptic, the operators  $\widetilde{D}\#A$ and
$\mathcal{D}$ are Fredholm operators in suitable spaces of functions and one has:
\begin{equation}\label{eq-eq1}
 \ind D=\ind \widetilde{D}\#A=\ind \mathcal{D}.
\end{equation}
The index $\ind \mathcal{D}$ was expressed in the cited paper by an Atiyah--Singer type   formula, which we now describe.

\subsection{The Index of the $\psi$DO $\mathcal{D}$}\label{sec2}

The space of $\End\mathcal{E}$-valued differential forms on the cosphere bundle
$S^*M_{\mathbb{Z}}$ of a sufficiently large radius $R$ is denoted by
$\Omega(S^*M_{\mathbb{Z}} ,\End\mathcal{E})$.
This algebra is endowed with the differential  $d$ satisfying the Leibniz rule
(the differential is well defined, because   $\mathcal{E}$ is a flat bundle).
The operation of taking the trace of an operator in   $l^2(\mathbb{Z})$ gives
a (partially defined)  mapping
$$
\tr_\mathcal{E}: \Omega(S^*M_{\mathbb{Z}},\End\mathcal{E})\lra \Omega(S^*M_{\mathbb{Z}}).
$$
{\em The topological index} of an elliptic symbol $\sigma=\sigma(\mathcal{D})$ is the number
\begin{equation}\label{eq-indt2}
 \ind_t\sigma =\sum_j C_j\int_{S^*  M_{\mathbb{Z}}}\tr_\mathcal{E}
\left[(\sigma ^{-1}d\sigma)^{2j-1}\Td(T^*_\mathbb{C}M) \right]_{top},
\end{equation}
where $[\ldots]_{top}$ stands for the top degree component (that is, of degree $2\dim M+1$). We suppose here
that  $M_\mathbb{Z}$ is endowed with a metric of the form   $h+dt^2$,
where $h$ is a $g$-invariant metric on $M$. In this case, because of the fact that  $g$ is isometric,
this metric is well defined and we have   equality
$$
\Td(T^*_\mathbb{C}M)=\Td(T^*_\mathbb{C}M_\mathbb{Z})
$$
of differential forms that represent the Todd classes of the complexified cotangent bundles
of the manifolds   $M$ and $M_\mathbb{Z}$.

It is proven in the paper  \cite{SaSchSt2}  that  \eqref{eq-indt2} is well defined and we have the equality
\begin{equation}\label{ind-luke1}
 \ind \mathcal{D}=\ind_t \sigma(\mathcal{D}).
\end{equation}
By \eqref{eq-eq1} and \eqref{ind-luke1}, to prove Theorem \ref{th1} it suffices to prove the equality
$$
\ind_t\sigma(D)=\ind_t \sigma(\mathcal{D})
$$
of the topological indices of the original operator with shifts $D$ and the $\psi$DO   $\mathcal{D}$.

\section{The Topological Indices of the Operators with Shifts on  
$M$ and $M\times   \mathbb{R}$  are Equal}

Similarly as in Section~\ref{sect1}, we define the topological index for operators of the form
\begin{equation}\label{eq-1-1}
B=\sum_k B_k\left(x,-i\frac\partial {\partial x},t,-i\frac\partial {\partial t}\right)\widetilde{T}^k:C^\infty(M\times\mathbb{R})\longrightarrow C^\infty(M\times\mathbb{R})
\end{equation}
on the product $M\times \mathbb{R}$,  where the shift operator  $\widetilde{T}$
was defined above. Here the operators $B_k$ are differential operators in the
variables $x,t$ with coefficients polynomial in $t$ and smooth in $x$.  The
operator $\widetilde{D}\# A$ is an example of operators of the form
\eqref{eq-1-1}.

The symbol  $\sigma(B)$ of the operator $B$ is a smooth operator-valued function on the cosphere bundle
$S(T^*M\times \mathbb{R}^2)$ and is considered as the element (see \cite{SaSchSt1,SaSchSt2})
\begin{eqnarray}\label{eq-1-2}
\sigma(B) \in C^\infty(S(T^*M\times \mathbb{R}^2),\mathcal{S})^\mathbb{Z},\\ \nonumber \sigma(B)=
\sum_k \sigma(B_k)(\partial g^{n}(x,\xi),t,\tau)\mathcal{T}^k.
\end{eqnarray}
Note that, by contrast with \eqref{eq-1-1},  here the group action  is trivial along the variables
 $t,\tau$.

By  analogy with the above, we define a cycle over the algebra
$$
 C^\infty(S(T^*M\times \mathbb{R}^2),\mathcal{S})^\mathbb{Z}
$$
and the corresponding topological index
\begin{equation}\label{eq-indt4}
\ind_t: K_1\left(C^\infty(S(T^*M\times \mathbb{R}^2),\mathcal{S})^\mathbb{Z}\right)\lra \mathbb{C}.
\end{equation}
Moreover, to define this topological index, we use the same differential form $\Td(T^*_\mathbb{C}M)$ as in
\eqref{eq-sled1}.

\begin{proposition}
The following equality holds:
\begin{equation}\label{eq-mu3}
\ind_t \sigma(D)=\ind_t \sigma(\widetilde{D}\# A).
\end{equation}
\end{proposition}
\begin{proof}
Indeed, we have
\begin{equation}\label{eq-mu1}
 \sigma(\widetilde{D}\# A)= \sigma(D)\# \sigma(A),
\end{equation}
where on the right-hand side we have the external product of the symbol   $\sigma(D)$ and the $\mathbb{Z}$-invariant symbol $\sigma(A)$ (cf. \eqref{eq-product1}).

The topological index \eqref{eq-indt4} has the following  multiplicative property
\begin{equation}\label{eq-mu2}
\ind_t (\sigma(D)\# \sigma(A))=(\ind_t \sigma(D))(\ind_t \sigma(A))
\end{equation}
with respect to the external product. The proof of this property is standard
(e.g., see \cite{NaSaSt17}, p.\ 107, Lemma 9.10), and we omit the details here.
Further, we have $\ind_t \sigma(A)=1$, since $[\sigma(A)]\in K^1(\mathbb{S}^1)=K^0(\mathbb{R}^2)$
is the Bott element.  Using \eqref{eq-mu1} and \eqref{eq-mu2} we obtain the desired equality \eqref{eq-mu3}.
\end{proof}

\section{The Topological Indices of the Operator with Shifts on   $M\times\mathbb{R}$
and $\psi$DO on  $M_\mathbb{Z}$ are Equal}\label{sec3}

The purpose of this section is to obtain the equality
\begin{equation}\label{eq-mu4}
\ind_t \sigma(\widetilde{D}\# A)=\ind_t \sigma(\mathcal{D})
\end{equation}
of the topological indices of the operator  with shifts $\widetilde{D}\# A$
on $M\times\mathbb{R}$ and the $\psi$DO $\mathcal{D}$  on the quotient
$M_\mathbb{Z}=(M\times\mathbb{R})/\mathbb{Z}$.

\subsection{A Reduction to the Unit Sphere Bundle}

The left and right-hand sides in the desired equality \eqref{eq-mu4}  are defined by the symbols
$\sigma(\widetilde{D}\#A)$ and $\sigma(\mathcal{D})$, respectively. By construction, the components of
these symbols satisfy the relation
\begin{equation}\label{eq-78}
 \sigma_{ij}(\mathcal{D})=\alpha^{i}\sigma_{ij}(\widetilde{D}\#A),
\end{equation}
where  $\alpha$ is the mapping induced by the diffeomorphism
$$
T^*M\times \mathbb{R}^2\to T^*M\times \mathbb{R}^2, \quad
(x,\xi,t,\tau)\mapsto (x,\xi,t+1,\tau).
$$
The  relation \eqref{eq-78} is straightforward. Indeed, if   $a=\sum a_k(n)\mathcal{T}^k$, then $a_{ij}=a_{j-i}(i)$. Hence
$$
 \sigma_{ij}(\widetilde{D}\# A)=\sigma_{ij}(D)\# \sigma(A)=\sigma(D_{j-i})(\partial g^i(x,\xi))\# \sigma(A)(t,\tau)\delta_{i-j}.
$$
On the other hand, $\sigma_{ij}(\mathcal{D})= \sigma(D_{j-i})(\partial g^i(x,\xi))\# \sigma(A)(t+i,\tau)\delta_{i-j}$.

Further, the topological index of the symbol   $\sigma(\mathcal{D})$ is determined in terms of integration over the cosphere bundle of a sufficiently large radius   $R$. Consider the mapping  $h_R:T^*M_\mathbb{Z}\lra T^*M_\mathbb{Z}$
$$
h_R(x,\xi,t,\tau)=(x,R\xi,t,R\tau),
$$
which takes the unit spheres to the spheres of radius   $R$. Thus, the topological index of the symbol
$\sigma=\sigma(\mathcal{D})$ is equal to
\begin{equation}\label{eq-indt2qq}
 \ind_t \sigma =\sum_j C_j\int_{S^*  M_{\mathbb{Z}}}\tr_\mathcal{E}
\left[(h_R^*\sigma ^{-1}d(h_R^*\sigma))^{2j-1}\Td(T^*_\mathbb{C}M) \right]_{top},
\end{equation}
where we integrate over the unit sphere bundle. The right-hand side  in \eqref{eq-indt2qq}
does not depend on   $R$, since the topological index is homotopy invariant.
On the other hand, we will show that the expression on the right-hand side in \eqref{eq-indt2qq}
tends to $\ind_t(\sigma(\widetilde{D}\#A))$  as $R\to\infty$. This will give the desired equality \eqref{eq-mu4}.

To study the mentioned limit, we first study in detail the mapping, which takes
a symbol of the form  $\sigma(\widetilde{D}\#A)$ to the symbol
$h^*_R\sigma(\mathcal{D})_{S^*M_\mathbb{Z}}$, and then, at the end of the
section, we prove equality  \eqref{eq-mu4}.

\subsection{An Asymptotic Homomorphism}

Consider the family of mappings
 \begin{equation}\label{eq-star2}
\begin{array}{ccc}
 \Omega  (S(T^*M\times \mathbb{R}^{2 }), \mathcal{S})^\mathbb{Z}&\lra &  \Omega(S^*M_{\mathbb{Z} },\End\mathcal{E})\vspace{1mm}\\
\omega=\{\omega_{ij}\} & \longmapsto & \omega _R=\left\{ (h^*_R\alpha^{i}\omega_{ij})|_{S^*M_{\mathbb{Z} }}
\right\},
\end{array}
\end{equation}
with  parameter $R\ge 1$.  
Here $\alpha$ is the mapping of differential forms induced by the diffeomorphism
$(x,\xi,t,\tau)\mapsto (x,\xi,t+1,\tau)$. 
In \eqref{eq-star2} and below we treat differential forms on the
cosphere bundle as homogeneous differential forms on the ambient vector bundle.

The mapping  \eqref{eq-star2} is well defined for all $R$. 
Indeed, the smoothness of the resulting operator-valued
form  $\omega_R$ is proved by a direct computation; its invariance with respect to the action
 \eqref{eq-mu5} follows, because $\omega$ is  $\mathbb{Z}$-invariant.
We have
$$
 d(\omega_R)=( d \omega)_R, \qquad \forall \omega\in \Omega  (S(T^*M\times \mathbb{R}^{2 }), \mathcal{S}),
$$
which is easy to check, using the fact that the exterior differential   $d$ commutes with the induced mappings  $h^*_R,$ $\alpha^{n}$, and also with the restriction $\cdot|_{S^* M_\mathbb{Z}}$.

\begin{remark}The mapping \eqref{eq-star2} on functions can be written  more explicitly:
\begin{equation*}\label{eq-star1}
\begin{array}{ccc}
 C^\infty(S(T^*M\times \mathbb{R}^{2}), \mathcal{S})^\mathbb{Z}  &\lra &  C^\infty(S^*M_{\mathbb{Z}},\End\mathcal{E}),\vspace{1mm}\\
f=\{f_{ij}(x,\xi,t,\tau)\} & \longmapsto & f_R= \{f_{ij}(x,R\xi,t+i,R\tau)\}.
\end{array}
\end{equation*}
\end{remark}

Let $U\subset M$ be a coordinate neighborhood. In the domain $T^*U\times \mathbb{R}^{2}$ with the coordinates $x,\xi,t,\tau$ consider the space of
$\mathcal{S}$-valued forms  homogeneous of degree zero and smooth outside the zero section.  Denote this space by
$\Omega(T^*U\times \mathbb{R}^{2}_{t,\tau},\mathcal{S})$ and define the following decomposition of this space:
\begin{equation}\label{eq-deco1}
\Omega(T^*U\times \mathbb{R}^{2},\mathcal{S})= \bigoplus_{i+j\le n+2}\Omega_{i,j},\quad n=\dim M,
\end{equation}
where $\Omega_{i,j}$ is the subspace of forms that  contain  $i$ differentials $dt$ ($i=0,1$) and
$j$ differentials $d\xi$, $d\tau$  $(0\le j\le n+1)$. Obviously, we have
$
 \Omega_{i,j}\cdot\Omega_{i',j'}\subset \Omega_{i+i',j+j'}.
$

An operator family $a(R)$  with parameter $R\ge 1$, which acts in the space of sequences,
is called  {\em a family of order $k$}   if for each fixed $s$ it
is bounded as an operator family acting in the spaces
$$
l^2(\mathbb{Z},\mu_{s,R})\to l^2(\mathbb{Z},\mu_{s-k,R}),
$$
uniformly in $R\ge 1$. 
The norms in these spaces are defined by the family of measures $\mu_{s,R}(n)=(n^2+R^2)^{s}$.

\begin{proposition}\label{lem-3}
The mapping \eqref{eq-star2} has the following properties:
\begin{enumerate}
\item  Given $a\in\Omega_{i,j} $, we have
\begin{equation}\label{eq-est3}
 a_R=R^j\times\left(\text{family of order $-i-j$}\right).
\end{equation}
\item This mapping is an asymptotic homomorphism of algebras in the following sense:  for any $a\in\Omega_{i,j} $ and $a'\in\Omega_{i',j'} $ we have
\begin{equation}\label{eq-est4}
 (a'a)_R-(a'_R)(a_R)=R^{j+j'}\times\left(\text{family of order $-i-i'-j-j'-1$}\right).
\end{equation}
\end{enumerate}
\end{proposition}

\begin{proof}
1. Let us prove \eqref{eq-est3}.  We have
\begin{equation}\label{eq-aaa1}
a=\sum_k a_k\mathcal{T}^k,\qquad a_k\in\Omega(T^*U\times\mathbb{R}^2, l^\infty(\mathbb{Z})),
\end{equation}
where $a_k$ is treated as an operator family acting on the space $l^2(\mathbb{Z})$. 
In addition,
the condition $a\in \Omega(T^*U\times \mathbb{R}^{2},\mathcal{S})$ implies that the sequence of norms of the
elements $ a_k $ tends to zero faster than any power of  $k$ as
$k\to \infty$. We have
\begin{equation}\label{eq-qqq1}
a_R=\sum_k (a_k)_R \mathcal{T}^k.
\end{equation}
Let us estimate the norms of the operators in this expression. First, we have
\begin{equation}\label{eq-tt1}
\| \mathcal{T}^k \|_{l^2(\mathbb{Z},\mu_{s,R})\to l^2(\mathbb{Z},\mu_{s,R}) }\le C_s (1+k^2)^{|s|/2}.
\end{equation}
This follows from the estimate
\begin{equation}\label{eq-est7}
\sup_{n\in \mathbb{R}} \left(\frac{(n+k)^2+R^2}{n^2+R^2}\right) \le C  (1+k^2)
\end{equation}
uniformly in  $R\ge 1$ and $k\in\mathbb{R}$. 
This estimate  is easy to obtain by a direct computation.
Second, denote the element  $a_k$   by $b$ for brevity. We get
\begin{equation}\label{eq-omegaR1}
 b_R(n)=h^*_R \alpha^{n}b(n) .	
\end{equation}
Denote the norm in the space of differential forms on
$S^*M_\mathbb{Z}$ with continuous coefficients (that is, the maximum of the modulus of the coefficients of the form)
by $\|\cdot\|$. We will now estimate the norm of the expression \eqref{eq-omegaR1} in the local coordinates $x,t,\xi,\tau$. We write 
$$
b= \sum_{|J|+k=j} dt^i d\xi^J d\tau^k b_{ Jk}(t,\xi,\tau),
$$
where the form $b_{Jk}(t,\xi,\tau)$ has only differentials  $dx$ and is homogeneous in
$(t,\xi,\tau)$ of degree $-i-j$.  Substituting this expression in \eqref{eq-omegaR1}, we obtain the following estimate
\begin{equation}\label{eq-39}
 \|b_R(n)\| \le  C R^{j}(n^2+R^2)^{-i-j}\sum_{|J|+k=j}\|b_{Jk}\| \le C R^{j}(n^2+R^2)^{-i-j} \|b\|
\end{equation}
where $t\in[0,1]$.

Let us go back to  \eqref{eq-qqq1}. Each term in this sum is of the form
$$
 R^j\times\left(\text{family of order $-i-j$}\right).
$$
Moreover,  by the estimates  \eqref{eq-tt1} and \eqref{eq-39},
and the rapid decay of the coefficients   $a_k$, this series converges.
This shows that the corresponding operator acting on the   $l^2$-spaces admits the desired representation \eqref{eq-est3}.

2. We will now prove \eqref{eq-est4}. Let $a=\sum a_k\mathcal{T}^k$,  $a'=\sum a'_l\mathcal{T}^l$.
We have
\begin{equation}\label{eq-h1}
a'_Ra_R=\sum_{kl} (a'_k)_R \mathcal{T}^k  (a _l)_R \mathcal{T}^l=
\sum_{kl} (a'_k)_R \bigl[ \mathcal{T}^k  (a _l)_R\mathcal{T}^{-k}\bigr ]\mathcal{T}^{k+l},
\end{equation}
and also
\begin{equation}\label{eq-h2}
(a' a)_R=\sum_{kl} (a'_k \mathcal{T}^k   a _l  \mathcal{T}^l)_R=
\sum_{kl} (a'_k)_R \bigl[   (\mathcal{T}^{k}a _l\mathcal{T}^{-k})_R \bigr ]\mathcal{T}^{k+l}.
\end{equation}
Hence, the expression on the left-hand side in~\eqref{eq-est4} can be written out as
\begin{equation}\label{eq-33}
 (a'a)_R-(a'_R)(a_R)=
\sum_{kl} (a'_k)_R \bigl[ (\mathcal{T}^{k}a _l\mathcal{T}^{-k})_R -  \mathcal{T}^k  (a _l)_R\mathcal{T}^{-k} \bigr]\mathcal{T}^{k+l} .
\end{equation}
Note that the series in \eqref{eq-h1} and \eqref{eq-h2} converge rapidly by the already proved item 1.

Denote the subspace of scalar forms similar to  $\Omega_{i,j}$ (see \eqref{eq-deco1})  by $\Omega_{i,j}(T^*U\times\mathbb{R}^2)$.
\begin{lemma}\label{lem-7}
We have the estimate
\begin{equation}\label{eq-34}
    \|h^*_R (\alpha^{n}-\alpha^{n+m})b \|\le C\frac{(1+m^2)^{(i+j+2)/2} R^j}{(n^2+R^2)^{(i+j+1)/2}}\|b\|
\end{equation}
uniformly in  $b\in \Omega_{i,j}(T^*U\times\mathbb{R}^2) $, $m\in\mathbb{Z}$ and $R\ge 1$.
Here on the left-hand side of the inequality we have the norm in the space of differential forms on  $S^*(U\times[0,1])$ with continuous coefficients,
while $\|b\|$ is a seminorm in the space  $\Omega(T^*U\times\mathbb{R}^2)$.
\end{lemma}
\begin{proof}
We will obtain the estimate  in the local coordinates $x,\xi,t,\tau$ on $T^*U\times T^*[0,1]$.  Let
$$
 b=\sum_{|J|+k=j}dt^i d\xi^Jd\tau^k b_{Jk}(t,\xi,\tau),
$$
where the form $b_{Jk}(t,\xi,\tau)$ has only differentials $dx$ and is homogeneous in $(t,\xi,\tau)$ of degree $-i-j$. We now get
\begin{multline}\label{eq-35}
\|h^*_R (\alpha^{n}-\alpha^{n+m})b\|\le\\
\le \sum_{|J|+k=j}R^{|J|+k}\bigl\|dt^id\xi^Jd\tau^k [b_{Jk}(t+n,R\xi,R\tau) - b_{Jk}(t+n+m,R\xi,R\tau)]\bigr\|\le\\
\le C R^{j}\sum_{|J|+k=j}\|b_{Jk}(t+n,R\xi,R\tau) - b_{Jk}(t+n+m,R\xi,R\tau)\|.
\end{multline}
Let us estimate the norm on the right-hand side in \eqref{eq-35}.
Applying Lagrange's remainder formula and the inequality  \eqref{eq-est7}, we obtain
\begin{multline}\label{eq-36}
\|b_{Jk}(t+n,R\xi,R\tau) - b_{Jk}(t+n+m,R\xi,R\tau)\| \le \\
\le C \cdot |m|\left((n+t+\theta m)^2+R^2\right)^{-(i+j+1)/2}\|b_{Jk}\|\le
 C' \frac{ (1+m^2)^{(i+j+2)/2}}{(n^2+R^2)^{(i+j+1)/2}}\|b\|,
\end{multline}
where $\theta\in[0,1]$.  The inequalities \eqref{eq-35} and \eqref{eq-36} give the desired inequality \eqref{eq-34}.

The proof of Lemma~\ref{lem-7} is now complete.
\end{proof}

We now estimate the expression \eqref{eq-33} using \eqref{eq-34},   \eqref{eq-est3}, and the fact that the
coefficients $a'_k$ and $a_l$ decay rapidly.   We obtain the desired expression \eqref{eq-est4}.

The proof of Proposition~\ref{lem-3} is complete.
\end{proof}

\subsection{Proof of the Equality of the Topological Indices}

\begin{proposition}\label{prop13}
Given an elliptic symbol $\sigma(B)\in C^\infty(S(T^*M\times \mathbb{R}^2),\mathcal{S})^\mathbb{Z}$,
the corresponding symbol $(\sigma(B))_R\in C^\infty(S^*M_\mathbb{Z},\End\mathcal{E})$, which we will denote by
$\sigma(\mathcal{B})$, is elliptic and we have
$$
 \ind_t \sigma(B)=\ind_t\sigma(\mathcal{B}).
$$
\end{proposition}
\begin{proof}
1. For brevity we will use the notation $\sigma=\sigma(B)$, $\sigma(\mathcal{B})=\sigma_R$.
By Proposition~\ref{lem-3}  the families $\sigma_R$ and $(\sigma^{-1})_R$ are inverses of each other up to
families of order  $-1$. This implies that the family  $\sigma_R$ is invertible
for all sufficiently large $R$. Hence, the topological index
\begin{equation}\label{eq-indt2a}
 \ind_t \sigma(\mathcal{B})=\sum_j C_j\int_{S^* M_{\mathbb{Z}}}\tr_\mathcal{E}
\bigl[\left(\sigma_R^{-1}d\sigma_R\right)^{2j-1} \Td(T^*_\mathbb{C}M)\bigr]_{top}
\end{equation}
is well defined.

2. We claim that the following decomposition holds
\begin{equation}\label{eq-qq1}
\bigl[(\sigma_R^{-1}d\sigma_R)^{2j-1}\Td(T^*_\mathbb{C}M)\bigr]_{top}=\bigl[((\sigma^{-1}d\sigma)^{2j-1})_R \Td(T^*_\mathbb{C}M)\bigr]_{top}+\varepsilon_j,
\end{equation}
where the family $\varepsilon_j$ is of the form
\begin{equation}\label{eq-99a}
\varepsilon_j=R^{n}\times\left(\text{family of order $-2-n$}\right), n=\dim M.
\end{equation}
Indeed, let us prove \eqref{eq-qq1}  using Proposition~\ref{lem-3}. Note that it suffices to obtain this
equality in a neighborhood   $S^*(U\times[0,1])$, where $U\subset M$  is a
neighborhood with coordinates $x$.  We decompose the $1$-form  $d\sigma$ in $U$
$$
d\sigma=a+b+c\in \Omega_{0,0}\oplus\Omega_{1,0}\oplus\Omega_{0,1}
$$
according to the decomposition \eqref{eq-deco1}.  Then the form on the left-hand side in  \eqref{eq-qq1} is equal to
\begin{equation}\label{eq-99b}
\Bigl[  \big((\sigma^{-1})_R +\varepsilon) (a_R+b_R+c_R)\bigr)^{2j-1}\Td(T^*_\mathbb{C}M)\Bigr]_{top},
\end{equation}
where $\varepsilon$ is a family of order $-1$ (here we used Proposition~\ref{lem-3}, item 2).
Opening the brackets in \eqref{eq-99b},  we represent this expression as a finite sum of monomials,
each of which is a product of the elements  $(\sigma^{-1})_R$, $\varepsilon$, $a_R$, $b_R$, $c_R$,
and the class $\Td(T^*_\mathbb{C}M)$. Consider an arbitrary monomial of this form. Since we
consider only the top degree component of this monomial, there is exactly one factor   $b_R$
and exactly $n=\dim M$   factors $c_R$. This fact together with Proposition~\ref{lem-3}, item 1, imply
that if the monomial contains at least one factor   $\varepsilon$, then this monomial is of the form
\eqref{eq-99a}. We will not consider such monomials below. Further, in this monomial, we can
replace a product of the form  $x_Ry_R$ by the expression $(xy)_R$, because the difference
of these two expressions is a family of order less by one  (Proposition~\ref{lem-3}, item 2).
So, the family \eqref{eq-99b} is equal to
\begin{equation}\label{eq-99c}
\Bigl[\bigl((\sigma^{-1} (a+b+c))^{2j-1}\bigr)_R\Td(T^*_\mathbb{C}M)\Bigr]_{top}=
\Bigl[\bigl((\sigma^{-1}d\sigma)^{2j-1}\bigr)_R \Td(T^*_\mathbb{C}M)\Bigr]_{top}
\end{equation}
up to a family of the form \eqref{eq-99a}. Hence the proof of \eqref{eq-qq1} is now complete.

3. We now substitute \eqref{eq-qq1}  in \eqref{eq-indt2a} and obtain
\begin{multline}\label{eq-multi1}
\ind_t \sigma(\mathcal{B})= \sum_j C_j\int_{S^* M_{\mathbb{Z}}}\tr_\mathcal{E}
\left[\bigl[(\sigma^{-1}d\sigma)^{2j-1}\bigr]_R\Td(T^*_\mathbb{C}M)+\varepsilon_j \right]_{top}=\\
 =\sum_j C_j  \tau\bigl[\left(\sigma^{-1}d\sigma\right)^{2j-1}\bigr] +O(R^{-1})
= \ind_t \sigma(B)+O(R^{-1})\lra\ind_t \sigma(B).
\end{multline}
Here the second equality follows from Lemma~\ref{lem-3x}   below,
the third equality is just the definition of the topological index $\ind_t \sigma(B)$ (see~\eqref{eq-indt4}).
Finally, the limit is taken as $R\to\infty$. 
Because $\ind_t \sigma(\mathcal{B})$ does not depend on $R$,
we see that \eqref{eq-multi1}  yields the desired equality \eqref{eq-mu4}.

The proof of Proposition~\ref{prop13} is complete.
 \end{proof}

\begin{lemma}\label{lem-3x}
\begin{enumerate}
\item Given a family $a(R):l^2(\mathbb{Z})\to l^2(\mathbb{Z})$  of order $m<-1$,  we have
\begin{equation}\label{trace4a}
 \tr  a(R) =O\left(R^{m+1}\right),\quad \text{if } R\ge 1.
\end{equation}
\item Given a form $\omega\in  \Omega  (S(T^*M\times \mathbb{R}^{2 }), \mathcal{S})^\mathbb{Z}$,   we have
\begin{equation}\label{trace4}
 \int_{S^*M_{\mathbb{Z} }} \tr_\mathcal{E} \bigl[\omega_R\Td(T^*_\mathbb{C}M)\bigr]_{top}=\tau(\omega)\quad
\text{for all $R\ge 1$.}
\end{equation}
\end{enumerate}
\end{lemma}

\begin{proof}
1. Let us prove the estimate \eqref{trace4a}.  To this end, let $\{e_k\}$ be the standard orthonormal base in
 $l^2(\mathbb{Z})$ ($e_k(k)=1$ and $e_k(n)=0$ if $k\ne n$). Since the family $a(R)$ is of order $m$, we get
$$
 \|a(R)e_k\|_{l^2(\mathbb{Z},\mu_{-m,R})}\le C \|e_k\|_{l^2(\mathbb{Z})},\qquad \forall k.
$$
Hence, we obtain the inequality
$
 |(a(R)e_k,e_k)_{l^2(\mathbb{Z})} | (k^2+R^2)^{-m/2}\le C.
$
This gives the desired estimate of the trace:
\begin{equation}\label{eq-est1}
|\tr a(R)|\le \sum_k |(a(R)e_k,e_k)_{l^2(\mathbb{Z})}|\le C \sum_k  (k^2+R^2)^{m/2}.
\end{equation}
This series converges if $m<-1$, and its sum is easily estimated
\begin{equation}\label{eq-est2}
\sum_k  (k^2+R^2)^{m/2}\le C_1 \int_{\mathbb{R}} dk(k^2+R^2)^{m/2}=
C_2 R^{1+m}\int_0^\infty \frac {dr}{(1+r^2)^{-m/2}}= O(R^{1+m}).
\end{equation}
Equations \eqref{eq-est1}  and \eqref{eq-est2} give the desired estimate \eqref{trace4a}.

2. We decompose the form $\omega\in  \Omega  (S(T^*M\times \mathbb{R}^{2 }), \mathcal{S})^\mathbb{Z}  $
as
$$
\omega= \omega^{0}(t)+dt\omega^{1}(t),
$$
where the forms $\omega^{i}(t)$ do not contain $dt$.

One can show (cf. the proof of Proposition~\ref{prop13}, item 2) that
the top degree component of the form   $\omega_R\Td(T^*_\mathbb{C}M)$ takes
values in operators of order $-\dim M-1\le -2$. Hence, this component takes
values in trace class operators and the expression on the right-hand side in
\eqref{trace4} is well defined. Further, a direct computation gives the
equalities
\begin{multline}\label{eq-longest1}
 \int_{S^*M_{\mathbb{Z}}} \tr_\mathcal{E} \bigl[\omega_R\Td(T^*_\mathbb{C}M)\bigr]_{top}=
\sum_{i,k}\int_{S^*M_{\mathbb{Z}}}  \bigl[ h^*_R\alpha^{i }dt^k\omega^k_{i,i}(t) \Td(T^*_\mathbb{C}M)\bigr]_{top}=\\
=\int_{S^*M_{\mathbb{Z}}}\sum_i \bigl[ h^*_R\alpha^{i }dt\omega^{1}_{00}(t)\Td(T^*_\mathbb{C}M)\bigr]_{top}=
\int_{S^*M_{\mathbb{Z}}}\left[\sum_i   R^{-1} dt \omega^1_{0,0}\left(\frac{t+i}R\right)\Td(T^*_\mathbb{C}M)\right]_{top}=\\
=\int_{\mathbb{R}_t\times S(T^*M\times\mathbb{R}_\tau)} \left[    dt  \omega^1_{0,0}\left(t\right)\Td(T^*_\mathbb{C}M)\right]_{top}
=  \tau(\omega).
\end{multline}
Here the first equality follows from the definition of the mapping \eqref{eq-star2},
the second follows from the fact that  $\omega$ is $\mathbb{Z}$-invariant.
The third equality follows, because  $\omega$ is homogeneous in $\xi,t,\tau$. The fourth equality in   \eqref{eq-longest1} follows from the decomposition
$$
\mathbb{R}_t=\bigcup_{i} \frac 1{R}([0,1)+i),
$$
and the fifth equality follows from the definition of the trace   $\tau$ (see~\eqref{eq-sled1})
and the fact that the integrals of the form
\begin{equation}\label{eq-q8}
\sum_{k=0,1} dt^k\omega^k_{0,0}\left(t\right)\Td(T^*_\mathbb{C}M)
\end{equation}
over the hypersurfaces
$$
\mathbb{R}_t\times S(T^*M\times\mathbb{R}_\tau)\quad\text{and}\quad
S(T^* M\times \mathbb{R}^{2}_{t,\tau})\quad \text{in }T^* M\times \mathbb{R}^{2}_{t,\tau}
$$
are equal  (because the form \eqref{eq-q8} is homogeneous of order zero and is induced by a form on the  cosphere bundle $S(T^* M\times \mathbb{R}^{2}_{t,\tau})$).
\end{proof}

\subsection{The Proof of the Main Theorem}

Combining the equalities \eqref{eq-eq1}, \eqref{ind-luke1}, \eqref{eq-mu4}, \eqref{eq-mu3} we get the desired equality \eqref{eq-index1}:
$$
\ind D=\ind \widetilde{D}\# A=\ind \mathcal{D}=\ind_t \sigma(\mathcal{D})=\ind_t \sigma(\widetilde{D}\# A)=\ind_t \sigma(D).
$$
The proof of the Index Theorem~\ref{th1} is now complete.

 
 \vspace{8mm}

{\rm
\noindent Peoples' Friendship University of Russia and \newline
\noindent Institut f\"ur Analysis, Leibniz Universit\"at Hannover, Germany\newline 
\noindent antonsavin@mail.ru \vspace{5mm}
 
\noindent Institut f\"ur Analysis, Leibniz Universit\"at Hannover, Germany \newline
\noindent schrohe@math.uni-hannover.de \vspace{5mm}

\noindent Peoples' Friendship University of Russia and \newline
\noindent Institut f\"ur Analysis, Leibniz Universit\"at Hannover, Germany\newline 
\noindent sternin@mail.ru \vspace{5mm}}

\end{document}